\documentclass[12pt]{article}
\usepackage{amsfonts}
\usepackage{amssymb}
\usepackage{amsmath}
\usepackage[numbers,sort&compress]{natbib}
\usepackage[colorlinks,citecolor=blue,urlcolor=blue]{hyperref}
\usepackage{graphicx,indentfirst,tabularx}

\newtheorem{theorem}{Theorem}

\newtheorem{conjecture}[theorem]{Conjecture}

\newtheorem{definition}[theorem]{Definition}
\newtheorem{example}[theorem]{Example}

\newtheorem{lemma}[theorem]{Lemma}

\newenvironment{proof}[1][Proof]{\noindent\textbf{#1.} }{\ \rule{0.5em}{0.5em}}
\input{tcilatex}
\input{tcilatex}

\begin{document}

\title{The Taylor Measure and its Applications}
\author{Athanasios C. Micheas \\
Department of Statistics, University of Missouri, \\
146 Middlebush Hall, Columbia, MO 65211-6100, USA,\\
email: micheasa@missouri.edu}
\maketitle

\begin{abstract}
We propose and study a novel collection of signed measures, which will be
apply called Taylor measures. Stochastic versions of the new measures are
also defined and studied. We illustrate, through examples, how the
deterministic and stochastic versions of the proposed Taylor measures emerge
as a unifying framework that includes many concepts from mathematics and
probability theory as special cases.
\end{abstract}

\textbf{Mathematics Subject Classifications}: Primary 28A75, 60B05;
Secondary 46C05, 46N30, 60B11, 60E05

\textbf{Keywords}: Discrete Probability Measure, Hilbert Space, Polish
Space, Stochastic Taylor Measure, Taylor Measure, Taylor Probability Measures

\section{Introduction}

Taylor expansions provide a fundamental method of approximation in
mathematics and related fields, with many important applications emerging
over the years, most notably in numerical analysis. In particular, consider
a real valued function $f:\Re \rightarrow \Re $ that is sufficiently smooth
about a point $x_{0}\in \Re $. Then $f$ assumes a Taylor expansion at the
point $x\in \Re $ via 
\begin{equation}
f(x)=\sum_{n=0}^{+\infty }\frac{f^{(n)}(x_{0})}{n!}%
(x-x_{0})^{n}=T_{M,x_{0}}(x)+R_{M,x_{0}}(x),  \label{FullTaylorExp}
\end{equation}%
where $T_{M,x_{0}}(x)=\sum_{n=0}^{M}\frac{f^{(n)}(x_{0})}{n!}(x-x_{0})^{n}$
is the $M^{th}$ order Taylor polynomial and $R_{M,x_{0}}(x)$ the remainder
term with the property $R_{M,x_{0}}(x)=o(|x-x_{0}|^{M+1})$ as $%
|x-x_{0}|\rightarrow 0$.

Applications of Taylor{'}s theorem include numerical algorithms for
optimization (\cite{More1978}; \cite{Conn2000}), state estimation (\cite%
{Sarkka2013}, Ch. 5), ordinary differential equations (\cite{Hairer1993},
Ch. 2), Taylor expansions for vector valued functions (\cite{feng2014exact}%
), for solutions of stochastic partial differential equations (\cite%
{jentzen2010taylor}), and for hitting probabilities in Brownian motion (\cite%
{hobson1999taylor}) and approximation of exponential integrals in Bayesian
statistics (\cite{Raudenbush2000}). Algorithms based on Taylor{'}s
approximation that can be viewed as statistical inference problems include
spline interpolation (\cite{Diaconis1988}; \cite{Kimeldorf1970}), numerical
quadrature (\cite{Diaconis1988}; \cite{Karvonen2017}; \cite{Karvonenetal2018}%
, differential equations (\cite{Schober2014}; \cite{Schober2019}; \cite%
{Teymur2016}), and linear algebra (\cite{Cockayneetal2019a}; \cite%
{Hennig2015}).

On the other hand, measure theory is another fundamental concept of
mathematics, that formalizes the concept of length, area, integration and
probability, to name but a few, which can be essential in understanding the
behavior of functions and their approximations in various contexts. For
example, both measure theory and Taylor's theorem involve approximating
functions, e.g., integration wrt a measure can be viewed as a way to sum up
values of functions over a set, which can be approximated using Taylor
polynomials.

Moreover, Taylor's theorem is used in analysis to study functions that are
analytic, or more generally, measurable, and the convergence of Taylor
series can be studied using concepts from measure theory, particularly when
it comes to understanding the behavior of functions in different spaces. In
addition, several topics from real analysis and functional analysis, borrow
strength from both measure theory and Taylor's theorem, which have led to
methodologies with great consequences and advancements in the theory of
mathematics and statistics.

Recent texts on real analysis, measure and probability theory that study all
aspects of these classic methodologies and more include \cite%
{Royden1989,FristedtGray1997,fremlin2000measure,Vestrup2003,
Dudley2004,Billingsley2013, klenke2013probability,
kallenberg2017random,micheas2018theory,
durrett2019probability,galambos2023advanced}, and \cite%
{ghahramani2024fundamentals}. Some recent papers on measure and probability
theory include \cite{jonas2022generalized, acu2023discrete} for probability
measures and random variables, \cite{delladio2024superdensity} for Radon
measures in $\Re^p$, \cite{savare2021sobolev} for Sobolev spaces in extended
metric-measure spaces, and \cite{scholpple2025spaces} for reproducing kernel
Hilbert spaces.

In this paper, we exemplify this important interplay between measure theory
and Taylor's theorem, by introducing a new collection of signed measures
motivated by equation (\ref{FullTaylorExp}), that provides a unifying
framework in mathematics and probability theory. In particular, as we will
show thoughout the exposition below, various mathematical and probabilistic
concepts are special cases of this collection of measures. The unifying
framework we propose emerges as a generalization of Taylor's theorem,
provides approximations of analytic functions via random Taylor measures,
includes as special cases significant stochastic processes like Brownian
motion, martingales, random walks, time series models, and more.

The paper proceeds as follows; in Section 2 we introduce signed Taylor
measures and then study properties of the space of Taylor measures
extensively, in Section 3. We introduce the concepts of the positive and
negative Taylor probability measures and densities in Section 4. The
stochastic version of Taylor measures is introduced in Section 5, were we
illustrate that many well known stochastic processes are special cases. In
Section 6 we use Taylor measures to create a new space of functions, study
its properties and connect the proposed Taylor measures with Taylor's
theorem. Concluding remarks are given in the last section.

\section{Taylor Measures}

Let $B\in \mathcal{B}(\mathbb{N}),$ a Borel set of $\mathbb{N}=\{0,1,2,...\}$%
, and define a set function $T_{\gamma ,\mathbf{a}}:(\mathbb{N},\mathcal{B}(%
\mathbb{N}))\rightarrow (\Re ,\mathcal{B}(\Re ))$ by

\begin{equation}
T_{\gamma ,\mathbf{a}}(B)=\sum\limits_{n\in B}a_{n}\frac{\gamma ^{n}}{n!},
\label{TaylorMeasure}
\end{equation}%
where $\gamma \in \Re ,$ $\mathbf{a}=[a_{0},a_{1},a_{2},...],$ with $%
a_{n}\in \Re ,$ for all $n\in \mathbb{N}.$ Further assume that $T_{\gamma ,%
\mathbf{a}}(\varnothing )=0,$ for all $a_{n},\gamma \in \Re .$ Then it is
straightforward to prove the following.

\begin{theorem}[Taylor Measure]
The set function $T_{\gamma ,\mathbf{a}}(.)$ is a signed measure, and will
henceforth be called the Taylor measure. In addition, $T_{\gamma ,\mathbf{a}%
} $ is a $\sigma $-finite signed measure, and it becomes a finite signed
measure if one of the following conditions holds for the sequence $\mathbf{a}
$:\newline
a) $a_{n}$ are uniformly bounded, i.e., $|a_{n}|\leq M,$ $\forall n$, for
some $M>0.$\newline
b) $a_{n}$ are asymptotically equivalent to $Mb^{n}$, i.e., $a_{n}\thicksim
Mb^{n},$ for some $b,M\in \Re $.
\end{theorem}

\begin{proof}
By definition, $T_{\gamma ,\mathbf{a}}(\varnothing )=0$. Moreover, for any
collection of disjoint Borel sets $\{B_{i}\}$ of $\mathcal{B}(\mathbb{N}),$
we have%
\begin{equation*}
T_{\gamma ,\mathbf{a}}\left( \bigcup\limits_{i}B_{i}\right)
=\sum\limits_{n\in \bigcup\limits_{i}B_{i}}a_{n}\frac{\gamma ^{n}}{n!}%
=\sum\limits_{i}\sum\limits_{n\in B_{i}}a_{n}\frac{\gamma ^{n}}{n!}%
=\sum\limits_{i}T_{\gamma ,\mathbf{a}}(B_{i}),
\end{equation*}%
so that $T_{\gamma ,\mathbf{a}}$ is countably additive. Therefore, $%
T_{\gamma ,\mathbf{a}}$ is a signed measure. When $\gamma \geq 0,$ and $%
a_{n}\geq 0,$ for all $n\in \mathbb{N},$ we have $T_{\gamma ,\mathbf{a}%
}(B)\geq 0,$ $\forall B\in \mathcal{B}(\mathbb{N})$, and $T_{\gamma ,\mathbf{%
a}}$ is a measure.\newline
Now write $\mathbb{N}=\bigcup\limits_{i=0}^{+\infty }\{i\},$ with $T_{\gamma
,\mathbf{a}}(\{i\})=a_{i}\frac{\gamma ^{i}}{i!}<+\infty ,$ so that $%
T_{\gamma ,\mathbf{a}}$ is a $\sigma $-finite signed measure. Moreover,
under condition a), we have%
\begin{equation*}
T_{\gamma ,\mathbf{a}}(\mathbb{N})=\sum\limits_{n=0}^{+\infty }a_{n}\frac{%
\gamma ^{n}}{n!}<M\sum\limits_{n=0}^{+\infty }\frac{\gamma ^{n}}{n!}%
=Me^{\gamma }<+\infty .
\end{equation*}%
Under condition b), since $a_{n}\thicksim Mb^{n},$ $\forall \varepsilon >0,$ 
$\exists n_{0}>0,$ such that for all $n>n_{0},$ we have 
\begin{equation*}
\left\vert \frac{a_{n}}{Mb^{n}}-1\right\vert <\varepsilon \Rightarrow
a_{n}<Mb^{n}(\varepsilon +1),
\end{equation*}%
so that sending $\varepsilon $ to $0$ we obtain%
\begin{equation*}
T_{\gamma ,\mathbf{a}}(\mathbb{N})=\sum\limits_{n=0}^{+\infty }a_{n}\frac{%
\gamma ^{n}}{n!}<M(\varepsilon +1)\sum\limits_{n=n_{0}}^{+\infty }\frac{%
b^{n}\gamma ^{n}}{n!}=M\sum\limits_{n=0}^{+\infty }\frac{(b\gamma )^{n}}{n!}%
=Me^{b\gamma }<+\infty ,
\end{equation*}%
and therefore $T_{\gamma ,\mathbf{a}}$ becomes a finite signed measure.
\end{proof}

Following \cite{Dudley2004} (Corollary 5.6.2), the Lebesgue decomposition
and Radon-Nikodym theorems also hold for finite signed measures. The latter
theorem is particularly important since it leads to the following definition.

\begin{definition}[Taylor Derivative]
Consider a finite Taylor measure $T_{\gamma ,\mathbf{a}}$ defined on the
measurable space $(\mathbb{N},\mathcal{B}(\mathbb{N})),$ and let $\upsilon $
denote a $\sigma $-finite measure on $(\mathbb{N},\mathcal{B}(\mathbb{N})).$
Assume that $T_{\gamma ,\mathbf{a}}$ is absolutely continuous wrt $\upsilon $%
, denoted by $T_{\gamma ,\mathbf{a}}<<\upsilon $. Then, there exists a
measurable function $p_{T}:(\mathbb{N},\mathcal{B}(\mathbb{N}))\rightarrow
(\Re ,\mathcal{B}(\Re ))$, such that%
\begin{equation}
T_{\gamma ,\mathbf{a}}(B)=\int\limits_{B}p_{T}(x)\upsilon (dx),
\label{TaylorDensity}
\end{equation}%
for all $B\in \mathcal{B}(\mathbb{N})$. The Radon-Nikodym derivative $p_{T}$%
, denoted by $p_{T}=\left[ \frac{dT_{\gamma ,\mathbf{a}}}{d\upsilon }\right] 
$ a.e. wrt $\upsilon $, will henceforth be known as the Taylor derivative.
\end{definition}

A desirable choice for $\upsilon $ is counting measure, in which case
equation (\ref{TaylorDensity}) reduces to%
\begin{equation}
T_{\gamma ,\mathbf{a}}(B)=\sum\limits_{n\in B}p_{T}(n),
\label{TaylorMeasureviaDensity}
\end{equation}%
so that in view of (\ref{TaylorMeasure}), we have%
\begin{equation}
p_{T}(n)=T_{\gamma ,\mathbf{a}}(\{n\})=a_{n}\frac{\gamma ^{n}}{n!},
\label{TaylorDensityatn}
\end{equation}%
$n\in \mathbb{N}.$ We note that $p_{T}$ is not a density in the usual
statistical sense, i.e., a probability mass function with $p_{T}(n)\geq 0,$
and $\sum\limits_{n=0}^{+\infty }p_{T}(n)=1.$

Now, using Jordan decomposition theorem (e.g., \cite{micheas2018theory},
Theorem 3.8) for the signed measure $T_{\gamma ,\mathbf{a}}$, there exist
two mutually singular measures denoted by $T_{\gamma ,\mathbf{a}}^{+}$ and $%
T_{\gamma ,\mathbf{a}}^{-}$ such that%
\begin{equation*}
T_{\gamma ,\mathbf{a}}=T_{\gamma ,\mathbf{a}}^{+}-T_{\gamma ,\mathbf{a}}^{-},
\end{equation*}%
with this decomposition being unique. More precisely, if $\left\{
A^{+},A^{-}\right\} $ is a Hahn decomposition (e.g., \cite{micheas2018theory}%
, Theorem 3.7) of $T_{\gamma ,\mathbf{a}},$ we have%
\begin{eqnarray*}
T_{\gamma ,\mathbf{a}}^{+}(B) &=&T_{\gamma ,\mathbf{a}}(B\cap A^{+}),\text{
and} \\
T_{\gamma ,\mathbf{a}}^{-}(B) &=&-T_{\gamma ,\mathbf{a}}(B\cap A^{-}),
\end{eqnarray*}%
with $0\leq T_{\gamma ,\mathbf{a}}^{+}(B),T_{\gamma ,\mathbf{a}%
}^{-}(B)<+\infty ,$ $\forall B\in \mathcal{B}(\mathbb{N}).$

We will refer to the measures $T_{\gamma ,\mathbf{a}}^{+}$ and $T_{\gamma ,%
\mathbf{a}}^{-}$ as the positive and negative finite Taylor measures, which
in general signed measure theory are known as the upper and lower variations
of $T_{\gamma ,\mathbf{a}}$.

Denote the collection of all finite Taylor measures by%
\begin{equation}
\mathcal{T}^{\mathcal{F}}=\left\{ T_{\gamma ,\mathbf{a}}:T_{\gamma ,\mathbf{a%
}}(B)=\sum\limits_{n\in B}a_{n}\frac{\gamma ^{n}}{n!}\text{, }B\in \mathcal{B%
}(\mathbb{N}),\text{ }a_{n},\gamma \in \Re ,\text{ with }T_{\gamma ,\mathbf{a%
}}(\mathbb{N})<+\infty \right\} ,  \label{FiniteTaylorMeasureSpace}
\end{equation}%
and denote by $\mathcal{T}_{\gamma }^{\mathcal{F}}\subseteq \mathcal{T}^{%
\mathcal{F}},$ the collection of finite Taylor measures indexed by a fixed,
common, $\gamma \in \Re .$ Owing to the form of the signed measures $%
T_{\gamma ,\mathbf{a}}\in \mathcal{T}^{\mathcal{F}},$ it is natural to
consider the following conjecture, that will help us give some insight on
the structure of the space $\mathcal{T}^{\mathcal{F}}$.

\begin{conjecture}
\label{ConjectureNumbers}For any $a_{n},b,c_{n},d\in \Re $ and $n\in \mathbb{%
N}$, there exist $u_{n},\gamma \in \Re ,$ such that%
\begin{equation}
a_{n}b^{n}+c_{n}d^{n}=u_{n}\gamma ^{n},  \label{Conjecture}
\end{equation}%
for all $n\in \mathbb{N}.$
\end{conjecture}

\begin{proof}
For $n=1$, we have a real number $g=a_{1}b+c_{1}d,$ and there are infinite $%
u_{1},\gamma \in \Re ,$ that satisfy (\ref{Conjecture}), with $g=u_{1}\gamma 
$. Assume that (\ref{Conjecture}) holds for $n\in \mathbb{N}$. We prove it
holds for $n+1.$ We have%
\begin{equation*}
a_{n+1}b^{n+1}+c_{n+1}d^{n+1}=(a_{n+1}b)b^{n}+(c_{n+1}d)d^{n}=u_{n}\gamma
^{n}=\left( \frac{u_{n}}{\gamma }\right) \gamma ^{n+1},
\end{equation*}%
so that by induction the claim holds. Clearly, the $u_{n},\gamma \in \Re $
are not unique. Note that if the LHS in (\ref{Conjecture}) is non-zero we
must have $\gamma \neq 0,$ and if the LHS is zero we can simply choose $%
u_{n}=0,$ $\forall n\in \mathbb{N},$ and $\gamma \in \Re $.
\end{proof}

Note that the converse of the latter conjecture is trivially satisfied;
given $u_{n}$, $\gamma \in \Re $, take $b=d=\gamma ,$ $a_{n}=u_{n}^{+}=\max
\{0,u_{n}\},$ and $c_{n}=u_{n}^{-}=\max \{0,-u_{n}\}$.

In order to study properties of the space of these new measures we require
the concept of length within the space. The standard approach is to use the
total variation of the finite Taylor measure $T_{\gamma ,\mathbf{a}}$
defined by%
\begin{equation*}
\left\Vert T_{\gamma ,\mathbf{a}}\right\Vert (B)=T_{\gamma ,\mathbf{a}%
}^{+}(B)+T_{\gamma ,\mathbf{a}}^{-}(B)<+\infty ,
\end{equation*}%
$\forall B\in \mathcal{B}(\mathbb{N}).$ We investigate properties of $%
\mathcal{T}^{\mathcal{F}}$ next.

\section{Properties of the Space of Finite Taylor Measures}

\label{PropertiesFTM}From general theory of signed measures, we know that
the spaces of finite Taylor measures $\mathcal{T}^{\mathcal{F}}$ and $%
\mathcal{T}_{\gamma }^{\mathcal{F}},$ equipped with the total variation norm 
$\left\Vert T_{\gamma ,\mathbf{a}}\right\Vert ,$ are Banach spaces
(complete, linear, normed space). However, spaces of signed measures are not
Hilbert spaces (inner product, linear, complete, metric space) in general,
using total variation since it does not satisfy the parallelogram identity%
\begin{equation*}
\left\Vert T_{\gamma ,\mathbf{a}_{1}}^{(1)}+T_{\gamma ,\mathbf{a}%
_{2}}^{(2)}\right\Vert ^{2}+\left\Vert T_{\gamma ,\mathbf{a}%
_{1}}^{(1)}-T_{\gamma ,\mathbf{a}_{2}}^{(2)}\right\Vert ^{2}=2\left(
\left\Vert T_{\gamma ,\mathbf{a}_{1}}^{(1)}\right\Vert ^{2}+\left\Vert
T_{\gamma ,\mathbf{a}_{2}}^{(2)}\right\Vert ^{2}\right) ,
\end{equation*}%
even in $\mathcal{T}_{\gamma }^{\mathcal{F}}.$ Therefore, we require the
definition of a new inner product $\rho \left( T_{\gamma _{1},\mathbf{a}%
_{1}}^{(1)},T_{\gamma _{1},\mathbf{a}_{2}}^{(2)}\right) $ to equip $\mathcal{%
T}^{\mathcal{F}}$ with, in order to turn it into a Hilbert space, and then
using the induced norm, we will also have that it is a Banach space.

To that end, first we define a new map that will serve as the inner product
we need in order to study $\mathcal{T}^{\mathcal{F}}.$

\begin{lemma}
\label{InnerProductLemma}Consider $T_{\gamma _{1},\mathbf{a}%
_{1}}^{(1)},T_{\gamma _{1},\mathbf{a}_{2}}^{(2)}\in \mathcal{T}^{\mathcal{F}%
} $, and define the map $\rho :\mathcal{T}^{\mathcal{F}}\times \mathcal{T}^{%
\mathcal{F}}\rightarrow \Re ,$ by%
\begin{equation}
\rho \left( T_{\gamma _{1},\mathbf{a}_{1}}^{(1)},T_{\gamma _{2},\mathbf{a}%
_{2}}^{(2)}\right) (B)=\sum\limits_{n\in B}a_{n,1}a_{n,2}\frac{(\gamma
_{1}\gamma _{2})^{n}}{n!},  \label{InnerProduct}
\end{equation}%
where $\gamma _{1},\gamma _{2}\in \Re ,$ $\mathbf{a}%
_{k}=(a_{0,k},a_{1,k},a_{2,k},...),$ $a_{n,k}\in \Re ,$ $n\in \mathbb{N},$ $%
k=1,2,$ for any $B\in \mathcal{B}(\mathbb{N})$. Then, equipped with $\rho ,$
the space $\mathcal{T}^{\mathcal{F}}$ is an inner product space$.$
\end{lemma}

\begin{proof}
Let $T_{\gamma _{1},\mathbf{a}_{1}}^{(1)},T_{\gamma _{2},\mathbf{a}%
_{2}}^{(2)},T_{\gamma _{3},\mathbf{a}_{3}}^{(3)}\in \mathcal{T}^{\mathcal{F}%
},$ and take arbitrary $B\in \mathcal{B}(\mathbb{N}).$ First note that $\rho 
$ is symmetric since%
\begin{eqnarray*}
\rho \left( T_{\gamma _{1},\mathbf{a}_{1}}^{(1)},T_{\gamma _{2},\mathbf{a}%
_{2}}^{(2)}\right) (B) &=&\sum\limits_{n\in B}a_{n,1}a_{n,2}\frac{(\gamma
_{1}\gamma _{2})^{n}}{n!}=\sum\limits_{n\in B}a_{n,2}a_{n,1}\frac{(\gamma
_{2}\gamma _{1})^{n}}{n!} \\
&=&\rho \left( T_{\gamma _{2},\mathbf{a}_{2}}^{(2)},T_{\gamma _{1},\mathbf{a}%
_{1}}^{(1)}\right) (B),
\end{eqnarray*}%
and positive definite for non-zero $T_{\gamma _{1},\mathbf{a}_{1}}^{(1)},$
since%
\begin{equation*}
\rho \left( T_{\gamma _{1},\mathbf{a}_{1}}^{(1)},T_{\gamma _{1},\mathbf{a}%
_{1}}^{(1)}\right) (B)=\sum\limits_{n\in B}a_{n,1}^{2}\frac{\gamma _{1}^{2n}%
}{n!}>0.
\end{equation*}%
Moreover, for arbitrary $a,b\in \Re ,$ using (\ref{Conjecture}) we write%
\begin{equation*}
aa_{n,1}\gamma _{1}^{n}+ba_{n,2}\gamma _{2}^{n}=u_{n}^{(1,2)}\left( \gamma
_{1,2}\right) ^{n},
\end{equation*}%
for some $u_{n}^{(1,2)},\gamma _{1,2}\in \Re ,$ so that%
\begin{equation}
aT_{\gamma _{1},\mathbf{a}_{1}}^{(1)}(B)+bT_{\gamma _{2},\mathbf{a}%
_{2}}^{(2)}(B)=\sum\limits_{n\in B}(aa_{n,1}\gamma _{1}^{n}+ba_{n,2}\gamma
_{2}^{n})\frac{1}{n!}=\sum\limits_{n\in B}u_{n}^{(1,2)}\frac{\gamma
_{1,2}^{n}}{n!}.  \label{LinearForm}
\end{equation}%
As a result, we have linearity in the first argument since%
\begin{eqnarray*}
&&a\rho \left( T_{\gamma _{1},\mathbf{a}_{1}}^{(1)},T_{\gamma _{3},\mathbf{a}%
_{2}}^{(3)}\right) (B)+b\rho \left( T_{\gamma _{2},\mathbf{a}%
_{2}}^{(2)},T_{\gamma _{3},\mathbf{a}_{3}}^{(3)}\right) (B) \\
&=&\sum\limits_{n\in B}\left( aa_{n,1}a_{n,3}(\gamma _{1}\gamma
_{3})^{n}+ba_{n,2}a_{n,3}(\gamma _{2}\gamma _{3})^{n}\right) \frac{1}{n!} \\
&=&\sum\limits_{n\in B}\left( aa_{n,1}a_{n,3}\gamma
_{1}^{n}+ba_{n,2}a_{n,3}\gamma _{2}^{n}\right) a_{n,3}\frac{\gamma _{3}^{n}}{%
n!}=\sum\limits_{n\in B}u_{n}^{(1,2)}a_{n,3}\frac{(\gamma _{1,2}\gamma
_{3})^{n}}{n!} \\
&=&\rho \left( aT_{\gamma _{1},\mathbf{a}_{1}}^{(1)}+bT_{\gamma _{2},\mathbf{%
a}_{2}}^{(2)},T_{\gamma _{3},\mathbf{a}_{3}}^{(3)}\right) (B),
\end{eqnarray*}%
so that $\rho (.,.)$ defines an inner product in $\mathcal{T}^{\mathcal{F}}$.
\end{proof}

Next we prove linearity of $\mathcal{T}^{\mathcal{F}}$, which is known for
signed measure spaces, but the proof is presented here explicity.

\begin{lemma}
\label{LemmaLinearity}The space of finite Taylor measures $\mathcal{T}^{%
\mathcal{F}}$ is a linear (vector) space.
\end{lemma}

\begin{proof}
Take arbitrary $a,b\in \Re $ and let $T_{\gamma _{1},\mathbf{a}%
_{1}}^{(1)},T_{\gamma _{2},\mathbf{a}_{2}}^{(2)}\in \mathcal{T}^{\mathcal{F}%
} $. From the previous proof, using equation (\ref{LinearForm}) we have
trivially%
\begin{equation*}
aT_{\gamma _{1},\mathbf{a}_{1}}^{(1)}(B)+bT_{\gamma _{2},\mathbf{a}%
_{2}}^{(2)}(B)=\sum\limits_{n\in B}u_{n}^{(1,2)}\frac{\gamma _{1,2}^{n}}{n!}%
\in \mathcal{T}^{\mathcal{F}},
\end{equation*}%
for any $B\in \mathcal{B}(\mathbb{N}).$
\end{proof}

Based on the inner product of equation (\ref{InnerProduct}), we can
immediately equip $\mathcal{T}^{\mathcal{F}}$ with the induced norm%
\begin{equation}
\left\Vert T_{\gamma ,\mathbf{a}}\right\Vert _{\rho }=\sqrt{\rho \left(
T_{\gamma ,\mathbf{a}},T_{\gamma ,\mathbf{a}}\right) },  \label{TaylorNorm}
\end{equation}%
and distance%
\begin{equation}
d\left( T_{\gamma _{1},\mathbf{a}_{1}}^{(1)},T_{\gamma _{2},\mathbf{a}%
_{2}}^{(2)}\right) =\left\Vert T_{\gamma _{1},\mathbf{a}_{1}}^{(1)}-T_{%
\gamma _{2},\mathbf{a}_{2}}^{(2)}\right\Vert _{\rho }=\sqrt{\rho \left(
T_{\gamma _{1},\mathbf{a}_{1}}^{(1)}-T_{\gamma _{2},\mathbf{a}%
_{2}}^{(2)},T_{\gamma _{1},\mathbf{a}_{1}}^{(1)}-T_{\gamma _{2},\mathbf{a}%
_{2}}^{(2)}\right) }.  \label{TaylorDistance}
\end{equation}

Since $\left\Vert T_{\gamma ,\mathbf{a}}\right\Vert _{\rho }$ and $d\left(
T_{\gamma _{1},\mathbf{a}_{1}}^{(1)},T_{\gamma _{2},\mathbf{a}%
_{2}}^{(2)}\right) $ are defined based on an inner product, they are, by
definition, a norm and metric, respectively, so that $\mathcal{T}^{\mathcal{F%
}}$ becomes immediately a normed vector space and $\left( \mathcal{T}^{%
\mathcal{F}},d\right) $ is a metric space. Therefore, we have all the
important ingredients required for a Hilbert space, except for completeness,
which we collect next.

\begin{lemma}
The space of finite Taylor measures $\mathcal{T}^{\mathcal{F}},$ equipped
with the norm $\left\Vert .\right\Vert _{\rho }$, is complete.
\end{lemma}

\begin{proof}
Assume that the sequence of measures $v_{k}=T_{\gamma _{k},\mathbf{a}%
_{k}}^{(k)}\in \mathcal{T}^{\mathcal{F}},$ is Cauchy, and take an arbitrary $%
B\in \mathcal{B}(\mathbb{N}),$ with%
\begin{equation*}
v_{k}=\sum\limits_{n\in B}a_{n,k}\frac{\gamma _{k}^{n}}{n!},
\end{equation*}%
where we must have%
\begin{equation*}
\underset{k\rightarrow +\infty }{\lim }a_{n,k}=a_{n},
\end{equation*}%
and%
\begin{equation*}
\underset{k\rightarrow +\infty }{\lim }\gamma _{k}=\gamma ,
\end{equation*}
otherwise $v_{k}$ would not be Cauchy, i.e., it has to converge to a single
value. Then, by the Cauchy sequence definition, $\forall \varepsilon >0,$ $%
\exists N>0,$ such that, $\forall k,m>N,$ we have%
\begin{equation*}
\left\Vert v_{k}-v_{m}\right\Vert _{\rho }(B)<\varepsilon .
\end{equation*}%
We need to show that $v_{k},$ $k\in \mathbb{N},$ converges to an element of $%
\mathcal{T}^{\mathcal{F}}.$ Define $v(B)=\underset{k\rightarrow +\infty }{%
\lim }v_{k}(B),$ and write%
\begin{eqnarray*}
\left\Vert v_{m}-v\right\Vert _{\rho }(B) &=&\sqrt{\rho \left(
v_{m},v\right) (B)}=\sqrt{\rho \left( v_{m},\underset{k\rightarrow +\infty }{%
\lim }v_{k}\right) (B)}=\underset{k\rightarrow +\infty }{\lim }\sqrt{\rho
\left( v_{m},v_{k}\right) (B)} \\
&=&\underset{k\rightarrow +\infty }{\lim }\left\Vert v_{m}-v_{k}\right\Vert
_{\rho }(B)\leq \varepsilon ,
\end{eqnarray*}%
so that $v_{k}$ converges to $v$, and it remains to show that $v\in \mathcal{%
T}^{\mathcal{F}}.$ In particular, we have 
\begin{equation*}
v(B)=\underset{k\rightarrow +\infty }{\lim }v_{k}(B)=\underset{k\rightarrow
+\infty }{\lim }\sum\limits_{n\in B}a_{n,k}\frac{\gamma _{k}^{n}}{n!}%
=\sum\limits_{n\in B}\underset{k\rightarrow +\infty }{\lim }\left(
a_{n,k}\gamma _{k}^{n}\right) \frac{1}{n!}=\sum\limits_{n\in B}a_{n}\frac{%
\gamma ^{n}}{n!}\in \mathcal{T}^{\mathcal{F}},
\end{equation*}%
for any $B\in \mathcal{B}(\mathbb{N}),$ where we can swap the order of the
limit and summation signs via an appeal to the bounded convergence theorem.
\end{proof}

Combining all the results up to this point in this section, gives the
following important result for the space of finite Taylor measures.

\begin{theorem}[Hilbert space]
The space of finite Taylor measures $\mathcal{T}^{\mathcal{F}},$ equipped
with the inner product $\rho \left( .,.\right) ,$ is a Hilbert space.
\end{theorem}

This highly desirable property for $\mathcal{T}^{\mathcal{F}}$ allows us to
borrow strength from all the general results on Hilbert spaces in the
literature. For example, $\mathcal{T}^{\mathcal{F}}$ equipped with the norm $%
\left\Vert .\right\Vert _{\rho }$ is a Banach space, and following \cite%
{Dudley2004}, Theorems 5.4.7, 5.4.9 and Corollary 5.4.10, every Hilbert
space has an orthonormal basis. In particular, let $\mathcal{E}=\{e_{\gamma
_{i},\mathbf{a}_{i}}\}_{i\in I}$ denote an orthonormal basis of $\mathcal{T}%
^{\mathcal{F}},$ where $I$ is not necessarily countable, so that for any $%
T_{\gamma ,\mathbf{a}}\in \mathcal{T}^{\mathcal{F}},$ we have the
reproducing formula%
\begin{equation}
T_{\gamma ,\mathbf{a}}(B)=\sum_{i\in I}\rho \left( T_{\gamma ,\mathbf{a}%
},e_{\gamma _{i},\mathbf{a}_{i}}\right) (B)e_{\gamma _{i},\mathbf{a}_{i}}(B),
\label{ReproducingHilbertTmeasure}
\end{equation}%
for any $B\in \mathcal{B}(\mathbb{N}).$ Clearly, since $\mathcal{E}$ is not
unique, the latter representation of a signed Taylor measure is not unique.

There is one property that is not immediately acquired in a Hilbert space,
that of separability. If we can further show that there is a countable dense
subset of $\mathcal{T}^{\mathcal{F}}$, then $\mathcal{T}^{\mathcal{F}}$ will
be separable, and as a consequence, a Polish space (complete, separable,
metric space). We collect this result in the following.

\begin{theorem}[Polish Space]
\label{TaylorPolish}The Hilbert space of finite Taylor measures $\mathcal{T}%
^{\mathcal{F}},$ equipped with the induced norm $\left\Vert .\right\Vert
_{\rho },$ is a Polish space.
\end{theorem}

\begin{proof}
We have already seen that $\mathcal{T}^{\mathcal{F}}$ is a complete metric
space. It remains to show that there exists a countable dense subset. Take
arbitrary $B\in \mathcal{B}(\mathbb{N}),$ and consider any $T_{\gamma ,%
\mathbf{a}}\in \mathcal{T}^{\mathcal{F}},$ where%
\begin{equation*}
T_{\gamma ,\mathbf{a}}(B)=\sum\limits_{n\in B}a_{n}\frac{\gamma ^{n}}{n!},
\end{equation*}%
with $a_{n},\gamma \in \Re ,$ $n\in \mathbb{N}.$ Let $\mathcal{C}=\{T_{s,%
\mathbf{q}}\}\subset \mathcal{T}^{\mathcal{F}},$ the collection of all
finite Taylor measures with rational $s\in \mathbb{Q}$ and $\mathbf{q}\in 
\mathbb{Q}^{\mathbb{\infty }}$. Since the rationals $\mathbb{Q}$ are dense
in $\Re $, we can find sequences of rationals $\{s_{k}\}_{k=1}^{+\infty }$
and $\{q_{n,k}\}_{k=1}^{+\infty },$ such that $\underset{k\rightarrow
+\infty }{\lim }s_{k}=\gamma ,$ and $\underset{k\rightarrow +\infty }{\lim }%
q_{n,k}=a_{n},$ for all $n\in \mathbb{N}.$ Now define the collection of
Taylor measures $\mathcal{C}_{\lim }=\{\underset{k\rightarrow +\infty }{\lim 
}T_{s_{k},\mathbf{q}_{k}}\},$ where%
\begin{equation*}
T_{s_{k},\mathbf{q}_{k}}(B)=\sum\limits_{n\in B}q_{n,k}\frac{s_{k}^{n}}{n!}%
\in \mathcal{T}^{\mathcal{F}},
\end{equation*}%
with $\mathbf{q}_{k}=(q_{0,k},q_{1,k},q_{2,k},...)\in \mathbb{Q}^{\infty },$
and $s_{k}\in \mathbb{Q}$. Consequently, we can write%
\begin{eqnarray*}
\underset{k\rightarrow +\infty }{\lim }T_{s_{k},\mathbf{q}_{k}}(B) &=&%
\underset{k\rightarrow +\infty }{\lim }\sum\limits_{n\in B}q_{n,k}\frac{%
s_{k}^{n}}{n!}=\sum\limits_{n\in B}\underset{k\rightarrow +\infty }{\lim }%
\left( q_{n,k}\frac{s_{k}^{n}}{n!}\right) = \\
&=&\sum\limits_{n\in B}\left( \underset{k\rightarrow +\infty }{\lim }%
q_{n,k}\right) \frac{\left( \underset{k\rightarrow +\infty }{\lim }%
s_{k}\right) ^{n}}{n!}=T_{\gamma ,\mathbf{a}}(B),
\end{eqnarray*}%
so that $T_{\gamma ,\mathbf{a}}\in \mathcal{C}_{\lim }$. Therefore, the
closure of the countable set $\mathcal{C}$ is $\overline{\mathcal{C}}=%
\mathcal{C}\cup \mathcal{C}_{\lim }=\mathcal{T}^{\mathcal{F}}$, and $%
\mathcal{T}^{\mathcal{F}}$ is separable as desired.
\end{proof}

We discuss the topology of $\mathcal{T}^{\mathcal{F}}$ induced by the norm $%
\left\Vert .\right\Vert _{\rho },$ following the usual approach. First, we
define an open ball in $\mathcal{T}^{\mathcal{F}}$ by%
\begin{equation}
b(T_{\gamma ,\mathbf{a}},r)=\{T_{\gamma _{1},\mathbf{a}_{1}}:\left\Vert
T_{\gamma ,\mathbf{a}}-T_{\gamma _{1},\mathbf{a}_{1}}\right\Vert _{\rho
}(B)<r,\forall B\in \mathcal{B}(\mathbb{N})\},  \label{OpenBall}
\end{equation}%
and then define an open set $O\subset \mathcal{T}^{\mathcal{F}}$ as the set
with the property that $\forall T_{\gamma ,\mathbf{a}}\in O,$ $\exists r>0$,
such that $b(T_{\gamma ,\mathbf{a}},r)\subset O$. Finally, denote by $%
\mathcal{O}(\mathcal{T}^{\mathcal{F}})$ the collection of all open sets of $%
\mathcal{T}^{\mathcal{F}},$ so that the Borel sets of $\mathcal{T}^{\mathcal{%
F}}$ are easily defined by $\mathcal{B}(\mathcal{T}^{\mathcal{F}})=\sigma (%
\mathcal{O}(\mathcal{T}^{\mathcal{F}})),$ the generated $\sigma -$field from
the open sets of $\mathcal{T}^{\mathcal{F}}$. Consequently, the pair $(%
\mathcal{T}^{\mathcal{F}},\mathcal{B}(\mathcal{T}^{\mathcal{F}}))$ is a
measurable space, which can be equipped with a measure or a probability
measure. This construction will allow us to define and study important
applications of $\mathcal{T}^{\mathcal{F}}$, e.g., stochastic versions of
Taylor measures (random Taylor measure).

Next we consider a first consequence of the theoretical development up to
this point.

\section{Taylor Probability Measures and Densities}

As a first broad application to probability theory of the new collection of
measures $\mathcal{T}^{\mathcal{F}},$ we take a closer look at the positive
and negative Taylor measures. In particular, when $T_{\gamma ,\mathbf{a}%
}^{+}<<\upsilon $ and $T_{\gamma ,\mathbf{a}}^{-}<<\upsilon ,$ an appeal to
the Radon-Nikodym theorem twice yields the derivatives $p_{\gamma ,\mathbf{a}%
}^{+}=\left[ \frac{dT_{\gamma ,\mathbf{a}}^{+}}{d\upsilon }\right] $ and $%
p_{\gamma ,\mathbf{a}}^{-}=\left[ \frac{dT_{\gamma ,\mathbf{a}}^{-}}{%
d\upsilon }\right] ,$ which will be called the positive and negative Taylor
derivatives. Obviously, when all derivatives exist, we can write%
\begin{equation}
p_{T}=p_{\gamma ,\mathbf{a}}^{+}-p_{\gamma ,\mathbf{a}}^{-}.
\label{TaylorDerivativePosNeg}
\end{equation}

In addition, since $0\leq T_{\gamma ,\mathbf{a}}^{+}(\mathbb{N}),T_{\gamma ,%
\mathbf{a}}^{-}(\mathbb{N})<+\infty ,$ we can build proper, normalized
densities (probability mass functions) via%
\begin{equation}
f_{T}^{+}(n|\gamma ,\mathbf{a})=\frac{p_{\gamma ,\mathbf{a}}^{+}(n)}{%
T_{\gamma ,\mathbf{a}}^{+}(\mathbb{N})},  \label{PositiveTaylorDensity}
\end{equation}%
and%
\begin{equation}
f_{T}^{-}(n|\gamma ,\mathbf{a})=\frac{p_{\gamma ,\mathbf{a}}^{-}(n)}{%
T_{\gamma ,\mathbf{a}}^{-}(\mathbb{N})},  \label{NegativeTaylorDensity}
\end{equation}%
$n\in \mathbb{N},\ $where $\gamma $ and $\mathbf{a}$ can be thought of as
parameters that require estimation. From a statistical modeling point of
view, this allows us to build models for $f_{T}^{+}$ and $f_{T}^{-}$, and
perform simulation, as well as approximate a finite Taylor measure. The
exposition above leads to the following definition.

\begin{definition}[Taylor Probability Measures]
\label{TaylorProbMeasure}Let $T_{\gamma ,\mathbf{a}}\in \mathcal{T}^{%
\mathcal{F}}$, and assume that $T_{\gamma ,\mathbf{a}}^{+}<<\upsilon $ and $%
T_{\gamma ,\mathbf{a}}^{-}<<\upsilon ,$ where $\upsilon $ denotes counting
measure. The positive Taylor density $f_{T}^{+}$ is defined as the
normalized Radon-Nikodym derivative of $T_{\gamma ,\mathbf{a}}^{+}$ wrt $%
\upsilon $, and the negative Taylor density $f_{T}^{-}$ is defined as the
normalized Radon-Nikodym derivative of $T_{\gamma ,\mathbf{a}}^{-}$ wrt $%
\upsilon $. As a result, we define the positive Taylor probability measure by%
\begin{equation}
Q_{\gamma ,\mathbf{a}}^{+}(B)=\sum\limits_{n\in B}f_{T}^{+}(n|\gamma ,%
\mathbf{a}),  \label{PositiveTaylorPM}
\end{equation}%
and the negative Taylor probability measure by%
\begin{equation}
Q_{\gamma ,\mathbf{a}}^{-}(B)=\sum\limits_{n\in B}f_{T}^{-}(n|\gamma ,%
\mathbf{a}),  \label{NegativeTaylorPM}
\end{equation}%
for all $B\in \mathcal{B}(\mathbb{N})$.
\end{definition}

The following result provides a connection between a finite Taylor measure
and the corresponding positive and negative Taylor probability measures. It
follows immediately by the definition of the measures involved.

\begin{theorem}
\label{TaylorProbDistMeasureConn}Let $T_{\gamma ,\mathbf{a}}\in \mathcal{T}^{%
\mathcal{F}}$ and $Q_{\gamma ,\mathbf{a}}^{+},$ $Q_{\gamma ,\mathbf{a}}^{-}$
the corresponding positive and negative Taylor probability measures. Then we
can write%
\begin{equation}
T_{\gamma ,\mathbf{a}}^{+}(B)=T_{\gamma ,\mathbf{a}}^{+}(\mathbb{N}%
)Q_{\gamma ,\mathbf{a}}^{+}(B),  \label{PosTaylorPM}
\end{equation}%
and%
\begin{equation}
T_{\gamma ,\mathbf{a}}^{-}(B)=T_{\gamma ,\mathbf{a}}^{-}(\mathbb{N}%
)Q_{\gamma ,\mathbf{a}}^{-}(B),  \label{NegTaylorPM}
\end{equation}%
so that%
\begin{equation}
T_{\gamma ,\mathbf{a}}(B)=T_{\gamma ,\mathbf{a}}^{+}(\mathbb{N})Q_{\gamma ,%
\mathbf{a}}^{+}(B)-T_{\gamma ,\mathbf{a}}^{-}(\mathbb{N})Q_{\gamma ,\mathbf{a%
}}^{-}(B),  \label{TaylorProbMeasures}
\end{equation}%
for all $B\in \mathcal{B}(\mathbb{N})$. Clearly, since $0<T_{\gamma ,\mathbf{%
a}}^{+}(\mathbb{N}),T_{\gamma ,\mathbf{a}}^{-}(\mathbb{N})<+\infty ,$ we
have $T_{\gamma ,\mathbf{a}}^{+}<<Q_{\gamma ,\mathbf{a}}^{+},$ $T_{\gamma ,%
\mathbf{a}}^{-}<<Q_{\gamma ,\mathbf{a}}^{-},$ $Q_{\gamma ,\mathbf{a}%
}^{+}<<T_{\gamma ,\mathbf{a}}^{+},$ and $Q_{\gamma ,\mathbf{a}%
}^{-}<<T_{\gamma ,\mathbf{a}}^{-}.$
\end{theorem}

In order to study general properties of $\mathcal{T}^{\mathcal{F}}$ we
worked with the signed measures, as in the previous section, but when it
comes to applying the theory created, we turn to choosing appropriate forms
for the positive and negative Taylor densities. In particular, equation (\ref%
{TaylorProbMeasures}) can be viewed from two directions; firstly, given a
Taylor measure $T_{\gamma ,\mathbf{a}}$, we wish to find the underlying
positive and negative Taylor densities $f_{T}^{+}$ and $f_{T}^{-},$ that
created this measure $T_{\gamma ,\mathbf{a}},$ and second, given two
discrete densities $f_{T}^{+}$ and $f_{T}^{-}$, we can use them to create a
specific measure $T_{\gamma ,\mathbf{a}}$. Since $f_{T}^{+}$ and $f_{T}^{-}$
are discrete probability densities over $\mathbb{N}$, we entertain a
flexible modeling choice in the following.

\begin{example}[Taylor measure via Taylor densities]
\label{ExTaylorDensity}Consider the power series family of probability mass
functions defined by%
\begin{equation}
f(n|\zeta ,\mathbf{b})=c(\zeta ,\mathbf{b})b_{n}\frac{\zeta ^{n}}{n!},
\label{PowerSeriesPMF}
\end{equation}%
$n\in \mathbb{N},$ where $\mathbf{b}=[b_{0},b_{1},b_{2},...],$ and assume
that the normalizing constant satisfies%
\begin{equation}
0<c(\zeta ,\mathbf{b})^{-1}=\sum\limits_{n=0}^{+\infty }b_{n}\frac{\zeta ^{n}%
}{n!}=T_{\zeta ,\mathbf{b}}(\mathbb{N})<+\infty ,
\label{NormConstantPosNegDens}
\end{equation}%
where $\zeta \geq 0,$ $b_{n}\geq 0,$ for all $n\in \mathbb{N}$. We define
the discrete probability measure corresponding to $f(n|\zeta ,\mathbf{b})$ by%
\begin{equation*}
Q_{\zeta ,\mathbf{b}}(B)=\sum\limits_{n\in B}f(n|\zeta ,\mathbf{b}),
\end{equation*}%
\newline
for all $B\in \mathcal{B}(\mathbb{N})$, with $Q$ absolutely continuous wrt
counting measure $\upsilon $, i.e., $f(n|\zeta ,\mathbf{b})=\left[ \frac{dQ}{%
d\upsilon }\right] $ ae wrt $\upsilon .$\newline
Now consider two densities from the family (\ref{PowerSeriesPMF}), $%
f_{1}(n|\zeta _{1},\mathbf{b}_{1})$ and $f_{2}(n|\zeta _{2},\mathbf{b}_{2}),$
which will be treated as the positive $f_{T}^{+}(n|\gamma ,\mathbf{a}),$ and
negative $f_{T}^{-}(n|\gamma ,\mathbf{a}),$ Taylor densities, respectively.
More precisely, assume that%
\begin{equation*}
f_{T}^{+}(n|\gamma ,\mathbf{a})=f_{1}(n|\zeta _{1},\mathbf{b}_{1})=c(\zeta
_{1},\mathbf{b}_{1})b_{1n}\frac{\zeta _{1}^{n}}{n!}=\frac{1}{T_{\gamma ,%
\mathbf{a}}^{+}(\mathbb{N})}b_{1n}\frac{\zeta _{1}^{n}}{n!},
\end{equation*}%
and 
\begin{equation*}
f_{T}^{-}(n|\gamma ,\mathbf{a})=f_{2}(n|\zeta _{2},\mathbf{b}_{2})=c(\zeta
_{2},\mathbf{b}_{2})b_{2n}\frac{\zeta _{2}^{n}}{n!}=\frac{1}{T_{\gamma ,%
\mathbf{a}}^{-}(\mathbb{N})}b_{2n}\frac{\zeta _{2}^{n}}{n!},
\end{equation*}%
so that $\gamma $ and $\mathbf{a}$ depend on $\zeta _{1},$ $\zeta _{2},$ $%
\mathbf{b}_{1},$ and $\mathbf{b}_{2}$, by construction, with $T_{\gamma ,%
\mathbf{a}}\in \mathcal{T}^{\mathcal{F}}$, given by%
\begin{equation*}
T_{\gamma ,\mathbf{a}}(B)=\sum\limits_{n\in B}a_{n}\frac{\gamma ^{n}}{n!}.
\end{equation*}%
Using equations (\ref{PositiveTaylorDensity}) and (\ref%
{NegativeTaylorDensity}), we can write 
\begin{equation*}
p_{\zeta _{1},\mathbf{b}_{1}}^{+}(n)=b_{1n}\frac{\zeta _{1}^{n}}{n!},
\end{equation*}%
and%
\begin{equation*}
p_{\zeta _{2},\mathbf{b}_{2}}^{-}(n)=b_{2n}\frac{\zeta _{2}^{n}}{n!},
\end{equation*}%
so that the Taylor derivative of equation (\ref{TaylorDerivativePosNeg})
becomes%
\begin{equation*}
p_{T}(n)=b_{1n}\frac{\zeta _{1}^{n}}{n!}-b_{2n}\frac{\zeta _{2}^{n}}{n!}.
\end{equation*}%
As a consequence, using equation (\ref{TaylorDensityatn}), we can connect $%
\gamma ,$ and $\mathbf{a}$ with $\zeta _{1},$ $\zeta _{2},$ $\mathbf{b}_{1},$
and $\mathbf{b}_{2},$ via the following equation%
\begin{equation*}
a_{n}\frac{\gamma ^{n}}{n!}=b_{1n}\frac{\zeta _{1}^{n}}{n!}-b_{2n}\frac{%
\zeta _{2}^{n}}{n!},
\end{equation*}%
so that%
\begin{equation}
a_{n}\gamma ^{n}=b_{1n}\zeta _{1}^{n}-b_{2n}\zeta _{2}^{n},
\label{Connectgammaalphas}
\end{equation}%
for all $n\in \mathbb{N}$. In view of Conjecture (\ref{ConjectureNumbers}),
given $\zeta _{1},$ $\zeta _{2}\geq 0,$ $b_{1n},$ $b_{2n}\geq 0,$ we can
find $\gamma $ $\in \Re ,$ and $\mathbf{a}=[a_{0},a_{1},...],$ $a_{n}\in \Re
,$ for all $n\in \mathbb{N},$ such that (\ref{Connectgammaalphas}) holds.
The converse is trivially satisfied; if $\gamma >0,$ take $\zeta _{1}=\zeta
_{2}=\gamma ,$ and $b_{1n}=a_{n}^{+}=\max \{0,a_{n}\},$ and $%
b_{2n}=a_{n}^{-}=\max \{0,-a_{n}\}$. When $\gamma <0,$ set $\zeta _{1}=\zeta
_{2}=-\gamma ,$ and $b_{1n}=\max \{0,(-1)^{n}a_{n}\},$ and $b_{2n}=\max
\{0,-(-1)^{n}a_{n}\}$.\newline
This example shows us exactly how we can create signed Taylor measures, via
the underlying positive and negative Taylor densities, since from equation (%
\ref{TaylorMeasureviaDensity}) we can write%
\begin{equation}
T_{\gamma ,\mathbf{a}}(B)=\sum\limits_{n\in B}\left( b_{1n}\frac{\zeta
_{1}^{n}}{n!}-b_{2n}\frac{\zeta _{2}^{n}}{n!}\right) ,
\label{TaylorviaDensities}
\end{equation}%
for all $B\in \mathcal{B}(\mathbb{N}).$
\end{example}

In view of the latter example and Definition \ref{TaylorProbMeasure}, we
prove a characterization of the positive Taylor probability measure.

\begin{theorem}[Discrete Probability Measure Representation]
Let $\upsilon $ denote counting measure.\label{DTaylorPMRepres} A set
function $Q:(\mathbb{N},\mathcal{B}(\mathbb{N}))\rightarrow \lbrack 0,1]$ is
a discrete probability measure with $Q<<\upsilon ,$ if and only if $Q$ is a
positive Taylor probability measure.
\end{theorem}

\begin{proof}
The if part is trivially satisfied by Definition \ref{TaylorProbMeasure}.
For the other direction, assume that $Q:(\mathbb{N},\mathcal{B}(\mathbb{N}%
))\rightarrow \lbrack 0,1]$ is a discrete probability measure, and wlog take
its support to be $\mathbb{N}$, i.e., $p_{n}=Q(\{n\})>0,$ $n\in \mathbb{N},$
with $\sum_{n\in \mathbb{N}}p_{n}=1,$ and $Q(B)=\sum_{n\in B}p_{n},$ for all 
$B\in \mathcal{B}(\mathbb{N}),$ since $Q<<\upsilon .$ Take arbitrary $B\in 
\mathcal{B}(\mathbb{N}),$ and write%
\begin{equation*}
Q(B)=\sum_{n\in B}p_{n}=\sum_{n\in B}\frac{n!}{\gamma ^{n}}p_{n}\frac{\gamma
^{n}}{n!}=T_{\gamma ,\mathbf{a}}(B),
\end{equation*}%
where $\mathbf{a}=[a_{0},a_{1},a_{2},...],$ $a_{n}=n!p_{n}/\gamma ^{n}>0$, $%
n\in \mathbb{N},$ and choose any $\gamma >0.$ Now since $\gamma ,a_{n}>0$,
we must have $T_{\gamma ,\mathbf{a}}^{-}(B)=0,$ for all $B\in \mathcal{B}(%
\mathbb{N})$, so that $T_{\gamma ,\mathbf{a}}(B)=T_{\gamma ,\mathbf{a}%
}^{+}(B),$ with $T_{\gamma ,\mathbf{a}}^{+}(\mathbb{N})=Q(\mathbb{N})=1.$
From equation (\ref{PosTaylorPM}), we have that $T_{\gamma ,\mathbf{a}%
}^{+}(B)=Q_{\gamma ,\mathbf{a}}^{+}(B),$ with the positive Taylor
probability mass function given by%
\begin{equation*}
f_{T}^{+}(n|\gamma ,\mathbf{a})=p_{n}=\frac{p_{\gamma ,\mathbf{a}}^{+}(n)}{%
T_{\gamma ,\mathbf{a}}^{+}(\mathbb{N})}=p_{\gamma ,\mathbf{a}}^{+}(n)=a_{n}%
\frac{\gamma ^{n}}{n!},
\end{equation*}%
and the claim holds.
\end{proof}

The following example illustrates explicitly the identifiability issues of
the space $\mathcal{T}^{\mathcal{F}}.$

\begin{example}[Poisson-Taylor Signed Measures]
As a special case of the previous example, consider two Poisson probability
measures with densities wrt counting measure which are special cases of (\ref%
{PowerSeriesPMF}). In particular, we take $b_{1n}=b_{2n}=1,$ for all $n\in 
\mathbb{N}$, and $\zeta _{1},$ $\zeta _{2}>0,$ so that%
\begin{equation*}
\sum\limits_{n\in B}\left( \frac{\zeta _{1}^{n}}{n!}-\frac{\zeta _{2}^{n}}{n!%
}\right) =\sum\limits_{n\in B}\frac{\zeta _{1}^{n}-\zeta _{2}^{n}}{n!}%
=\sum\limits_{n\in B}a_{n}\frac{\gamma ^{n}}{n!}=T_{1,\mathbf{a}}(B),
\end{equation*}%
with $\gamma =1,$ and $\mathbf{a}=(a_{0},a_{1},...),$ $a_{n}=\zeta
_{1}^{n}-\zeta _{2}^{n}\in \Re $. Clearly, $T_{1,\mathbf{a}}(\mathbb{N}%
)=e^{\zeta _{1}}-e^{\zeta _{2}}<+\infty .$ Note that the representation (\ref%
{TaylorviaDensities}) is not unique, since%
\begin{equation*}
\sum\limits_{n\in B}\frac{\zeta _{1}^{n}-\zeta _{2}^{n}}{n!}%
=\sum\limits_{n\in B}\left( 1-\left( \frac{\zeta _{2}}{\zeta _{1}}\right)
^{n}\right) \frac{\zeta _{1}^{n}}{n!}=T_{\zeta _{1},\mathbf{a}_{1}}(B),
\end{equation*}%
with $\gamma =\zeta _{1},$ and $\mathbf{a}_{1}=(a_{1,0},a_{1,1},...),$ $%
a_{1,n}=1-\left( \zeta _{2}/\zeta _{1}\right) ^{n}\in \Re $, and%
\begin{equation*}
\sum\limits_{n\in B}\frac{\zeta _{1}^{n}-\zeta _{2}^{n}}{n!}%
=\sum\limits_{n\in B}\left( \left( \frac{\zeta _{1}}{\zeta _{2}}\right)
^{n}-1\right) \frac{\zeta _{2}^{n}}{n!}=T_{\zeta _{2},\mathbf{a}_{2}}(B),
\end{equation*}%
with $\gamma =\zeta _{2},$ and $\mathbf{a}_{2}=(a_{2,0},a_{2,1},...),$ $%
a_{2,n}=\left( \zeta _{1}/\zeta _{2}\right) ^{n}-1\in \Re $. The resulting
signed Taylor measure $T_{1,\mathbf{a}}=T_{\zeta _{1},\mathbf{a}%
_{1}}=T_{\zeta _{2},\mathbf{a}_{2}},$ will be aptly called the
Poisson-Taylor signed measure.
\end{example}

One of the major consequences of Taylor probability measures is that it
allows us approximations in $\mathcal{T}^{\mathcal{F}}$ via statistical
simulation. We end this section with an illustrative example of this idea.

\begin{example}[Taylor Measure Approximation]
\label{TaylorMeasureApproxEx}By Definition (\ref{TaylorProbMeasure}) we can
easily approximate the value of the signed Taylor measure via simulation as
follows; consider two independent, discrete random variables $N_{1}\thicksim
f_{1}(n|\zeta _{1},\mathbf{b}_{1})$ and $N_{2}\thicksim f_{2}(n|\zeta _{2},%
\mathbf{b}_{2}),$ i.e., $N_{1},N_{2}:(\Omega ,\mathcal{A},P)\rightarrow (%
\mathbb{N},\mathcal{B}(\mathbb{N}))$, two measurable functions from a
probability space $(\Omega ,\mathcal{A},P)$ into the measurable space $(%
\mathbb{N},\mathcal{B}(\mathbb{N})),$ and generate two independent random
samples $n_{1,1},...,n_{1,L_{1}}$ $\overset{iid}{\thicksim }f_{1}(n|\zeta
_{1},\mathbf{b}_{1})$ and $n_{2,1},...,n_{2,L_{2}}\overset{iid}{\thicksim }%
f_{2}(n|\zeta _{2},\mathbf{b}_{2}).$ Now using (\ref{TaylorProbMeasures}) we
can write%
\begin{eqnarray*}
T_{\gamma ,\mathbf{a}}(B) &=&T_{\zeta _{1},\mathbf{b}_{1}}^{+}(\mathbb{N}%
)\sum\limits_{n\in B}f_{T}^{+}(n|\zeta _{1},\mathbf{b}_{1})-T_{\zeta _{2},%
\mathbf{b}_{2}}^{-}(\mathbb{N})\sum\limits_{n\in B}f_{T}^{-}(n|\zeta _{2},%
\mathbf{b}_{2}) \\
&=&T_{\zeta _{1},\mathbf{b}_{1}}^{+}(\mathbb{N})P(N_{1}\in B)-T_{\zeta _{2},%
\mathbf{b}_{2}}^{-}(\mathbb{N})P(N_{2}\in B) \\
&=&T_{\zeta _{1},\mathbf{b}_{1}}^{+}(\mathbb{N})\mathbb{E}\left[ I(N_{1}\in
B)\right] -T_{\zeta _{2},\mathbf{b}_{2}}^{-}(\mathbb{N})\mathbb{E}\left[
I(N_{2}\in B)\right] ,
\end{eqnarray*}%
so that using the Strong Law of Large Numbers (SLLN) we have%
\begin{equation*}
\frac{T_{\zeta _{1},\mathbf{b}_{1}}^{+}(\mathbb{N})}{L_{1}}%
\sum\limits_{i=1,...,L_{1}}I(n_{1,i}\in B)-\frac{T_{\zeta _{2},\mathbf{b}%
_{2}}^{-}(\mathbb{N})}{L_{2}}\sum\limits_{j=1,...,L_{2}}I(n_{2,j}\in B)%
\overset{a.s.}{\rightarrow }T_{\gamma ,\mathbf{a}}(B),
\end{equation*}%
for all $B\in \mathcal{B}(\mathbb{N}),$ where $\mathbb{E}(.)$ denotes
expectation. When the normalizing constants $T_{\zeta _{1},\mathbf{b}%
_{1}}^{+}(\mathbb{N})$ and $T_{\zeta _{2},\mathbf{b}_{2}}^{-}(\mathbb{N})$
are not known in closed form, they can be approximated as well using the
aforementioned random samples, which can be obtained even if the normalizing
constant is not known, e.g., rejection methods or Metropolis-Hastings
samplers. An alternative approach, which is more straightforward, is to write%
\begin{equation*}
T_{\zeta _{1},\mathbf{b}_{1}}^{+}(\mathbb{N})=\sum\limits_{n\in \mathbb{N}%
}b_{1n}\frac{\zeta _{1}^{n}}{n!}=\sum\limits_{n\in \mathbb{N}}e^{\zeta
_{1}}b_{1n}\frac{e^{-\zeta _{1}}\zeta _{1}^{n}}{n!},
\end{equation*}%
so that%
\begin{equation*}
\frac{e^{\zeta _{1}}}{L_{1}}\sum\limits_{i=1,...,L_{1}}b_{1n_{1,i}}\overset{%
a.s.}{\rightarrow }T_{\zeta _{1},\mathbf{b}_{1}}^{+}(\mathbb{N})
\end{equation*}%
where $n_{1,1},...,n_{1,L_{1}}$ $\overset{iid}{\thicksim }Poisson(\zeta
_{1}),$ and%
\begin{equation*}
T_{\zeta _{2},\mathbf{b}_{2}}^{-}(\mathbb{N})=\sum\limits_{n\in \mathbb{N}%
}b_{2n}\frac{\zeta _{2}^{n}}{n!}=\sum\limits_{n\in \mathbb{N}}e^{\zeta
_{2}}b_{2n}\frac{e^{-\zeta _{2}}\zeta _{2}^{n}}{n!},
\end{equation*}%
which yields%
\begin{equation*}
\frac{e^{\zeta _{2}}}{L_{2}}\sum\limits_{i=1,...,L_{2}}b_{2n_{2,i}}\overset{%
a.s.}{\rightarrow }T_{\zeta _{2},\mathbf{b}_{2}}^{-}(\mathbb{N}),
\end{equation*}%
where $n_{2,1},...,n_{2,L_{2}}$ $\overset{iid}{\thicksim }Poisson(\zeta
_{2}) $.
\end{example}

\section{Stochastic Taylor Measures}

As a second application of the new collection of measures $\mathcal{T}^{%
\mathcal{F}},$ we consider introducing stochasticity to the space $\mathcal{T%
}^{\mathcal{F}}.$ In what follows, let $(\Omega ,\mathcal{A},P)$ denote a
probability space.

\begin{definition}[Stochastic Taylor Measure]
\label{RandTaylorMeas}Consider the collection of finite Taylor measures $%
\mathcal{T}^{\mathcal{F}}$ defined on the measurable space $(\mathbb{N},%
\mathcal{B}(\mathbb{N}))$. Let $(\mathcal{T}^{\mathcal{F}},\mathcal{V})$
denote the measurable space of $\mathcal{T}^{\mathcal{F}},$ where $\mathcal{V%
}$ is defined as the smallest $\sigma $-field such that for each $B\in 
\mathcal{B}(\mathbb{N}),$ the map $U:\mathcal{T}^{\mathcal{F}}\rightarrow
\Re ,$ defined by $U(Q)=Q(B),$ $Q\in \mathcal{T}^{\mathcal{F}},$ is $%
\mathcal{V}$-measurable. Then, we define a stochastic (or random) Taylor
measure (STM) as the measurable map $X:(\Omega ,\mathcal{A},P)\rightarrow (%
\mathcal{T}^{\mathcal{F}},\mathcal{V}),$ where $X^{-1}(V)\in \mathcal{A},$
for all $V\in \mathcal{V}.$
\end{definition}

Since for any element $T_{\gamma ,\mathbf{a}}$ of $\mathcal{T}^{\mathcal{F}%
}, $ there exist by definition a sequence $\mathbf{a}=[a_{0},a_{1},...]\in
\Re ^{\infty },$ $a_{n}\in \Re ,$ and a parameter $\gamma \in \Re ,$ that
help define $T_{\gamma ,\mathbf{a}},$ a natural appoach to creating random
Taylor measures emerges; in particular, we consider introducing
stochasticity in the sequence $\mathbf{a}$ and parameter $\gamma ,$ such
that,%
\begin{equation}
X(\omega )(B)=T_{\gamma (\omega ),\mathbf{a}(\omega )}(B)=\sum\limits_{n\in
B}a_{n}(\omega )\frac{\gamma (\omega )^{n}}{n!},  \label{STMDef}
\end{equation}%
for all $\omega \in \Omega ,$ and $B\in \mathcal{B}(\mathbb{N}),$ will allow
us to create a variety of STMs, i.e., consider a random sequence $\mathbf{a}%
:(\Omega ,\mathcal{A},P)\rightarrow (\Re ^{\infty },\mathcal{B}(\Re ^{\infty
})),$ and random variable $\gamma :(\Omega ,\mathcal{A},P)\rightarrow (\Re ,%
\mathcal{B}(\Re ))$. Some examples are in order.

\begin{example}[Moments of STMs]
Assume that $\gamma =I_{A},$ with $A\in \mathcal{A},$ the indicator function
of the measurable set $A$, independent of the random variables $a_{n},$
assumed to be independent with means {$\mu _{na}$ and variances $\sigma
_{na}^{2}$}, for each $n\in \mathbb{N}$. For a fixed $B\in \mathcal{B}(%
\mathbb{N}),$ the STM becomes%
\begin{equation*}
X(\omega )(B)=\sum\limits_{n\in B}a_{n}(\omega )\frac{I_{A}^{n}(\omega )}{n!}%
=I_{A}(\omega )\sum\limits_{n\in B}\frac{a_{n}(\omega )}{n!},
\end{equation*}%
with the convention $0^{0}=1$, so that $X(B)$ is such that%
\begin{equation*}
\mu =\mathbb{E}X(B)=\mathbb{E}\left( I_{A}\sum\limits_{n\in B}\frac{a_{n}}{n!%
}\right) =P(A)\sum\limits_{n\in B}\frac{\mu _{na}}{n!},
\end{equation*}%
and%
\begin{eqnarray*}
\sigma ^{2} &=&Var\left( I_{A}\sum\limits_{n\in B}\frac{a_{n}}{n!}\right)
=\sum\limits_{n\in B}Var\left( I_{A}\frac{a_{n}}{n!}\right)
=\sum\limits_{n\in B}\frac{1}{n!}Var\left( I_{A}a_{n}\right) \\
&=&\sum\limits_{n\in B}\frac{1}{n!}\left( \mathbb{E}\left[ \left(
I_{A}a_{n}\right) ^{2}\right] -\left[ \mathbb{E}\left( I_{A}\right) \mathbb{E%
}\left( a_{n}\right) \right] ^{2}\right) \\
&=&\sum\limits_{n\in B}\frac{1}{n!}\left( P(A)\mathbb{E}\left[ a_{n}^{2}%
\right] -P(A)^{2}\mu _{na}^{2}\right) =\sum\limits_{n\in B}\frac{P(A)}{n!}%
\left[ \sigma _{na}^{2}+(1-P(A))\mu _{na}^{2}\right] .
\end{eqnarray*}
\end{example}

\begin{example}[Measurability and STMs]
\label{MeasurableSTMs}Consider the setup of the previous example, and
simplify further by taking $a_{n}=n!/|B|$, fixed random variables, with $%
|B|=card(B)$, the cardinality of $B$, the STM becomes the indicator $%
X(\omega )(B)=I_{A}(\omega ),$ for all $B\in \mathcal{B}(\mathbb{N}).$
Alternatively, and more generally, set $\gamma =1,$ and $%
a_{n}=n!c_{n}I_{A_{n}},$ with $\{A_{n}\}_{n\in B}$ a partition of $\Omega ,$ 
$A_{n}\in \mathcal{A},$ for some real constants $c_{n}$, so that the STM
becomes a simple random variable (in canonical form), i.e.,%
\begin{equation}
X(\omega )(B)=\sum\limits_{n\in B}c_{n}I_{A_{n}}(\omega ),  \label{SimpleSTM}
\end{equation}%
for all $\omega \in \Omega .$ As a result, measurable simple functions, the
main ingredient and building block of measure and probability theory, are
special cases of STMs. Recall that (e.g., \cite{micheas2018theory}, Theorem
3.4), for any measurable function $f:\Re \rightarrow \Re ,$ there exists a
monotone sequence of simple, measurable functions $\{f_{k}\}_{k=0}^{+\infty
},$ such that $f_{k}(\omega )\rightarrow f(\omega ),$ as $k\rightarrow
+\infty ,$ for all $\omega \in \Omega .$ Consequently, any measurable
function (in particular, random variables) can be expressed as a monotone
limit of a sequence of STMs. Note here that measurability of a function does
not imply that the function is analytic or continuous, most notably, the
indicator function $I_{A}$.
\end{example}

We collect some important examples under the Gaussian assumption next.

\begin{example}[Gaussian STMs]
\label{GaussianSTMs}Assuming that the parameters {$a_{n}$} and $\gamma $
follow Gaussian distributions, the resulting STMs will be naturally called
Gaussian STMs (GSTM). We present some special cases of equation (\ref{STMDef}%
) below, in order to appreciate the plethora of stochastic processes and
mathematical applications we can construct via STMs.

\begin{enumerate}
\item Normal random variables: Assume that $\gamma $ is a constant (random
variable) and {take $a_{n}\overset{iid}{\thicksim }N(\mu _{a},\sigma
_{a}^{2})$}. {This is the simplest way of introducing randomness into }$%
T_{\gamma ,\mathbf{a}}.$ Now if $B=\{0\},$ $X(\{0\})=a_{0}\thicksim N(\mu
_{a},\sigma _{a}^{2})$, so that the GSTM contains the univariate normal as a
special case. Letting $B\in \mathcal{B}(\mathbb{N}),$ we have $X(B)\thicksim 
$ $N\left( \mu _{a}\sum\limits_{n\in B}\frac{\gamma ^{n}}{n!},\sigma
_{a}^{2}\sum\limits_{n\in B}\frac{\gamma ^{2n}}{(n!)^{2}}\right) ,$ and if $%
B=\mathbb{N},$ $X(\mathbb{N})\thicksim N\left( \mu _{a}e^{\gamma },\sigma
_{a}^{2}\sum\limits_{n\in \mathbb{N}}\frac{\gamma ^{2n}}{(n!)^{2}}\right) .$

\item Analytic functions: In the previous example, when {$a_{n}\overset{indep%
}{\thicksim }N(\mu _{na},\sigma _{na}^{2})$}, we have%
\begin{equation*}
\mathbb{E}(X(\mathbb{N}))=\sum_{n=0}^{+\infty }\mu _{na}\frac{\gamma ^{n}}{n!%
},
\end{equation*}%
so that setting $\gamma (x)=x-x_{0},$ $x,x_{0}\in \Re ,$ and $\mu
_{na}=f^{(n)}(x_{0}),$ for some analytic function $f$ at $x_{0}$, we can
write%
\begin{equation*}
\mathbb{E}(X(\mathbb{N}))=\sum_{n=0}^{+\infty }\mu _{na}\frac{\gamma (x)^{n}%
}{n!}=\sum_{n=0}^{+\infty }\frac{f^{(n)}(x_{0})}{n!}(x-x_{0})^{n}=f(x),
\end{equation*}%
by Taylor's Theorem, and consequently, the GSTM can used to approximate
analytic functions.

\item Random walk {GSTMs}: A random sequence defined through sums of iid
random variables is a random walk. In particular, let $X_{k}\thicksim Q,$ $%
k\in \mathbb{N}^{+}=$ $\{1,2,...\},$ defined on $(\Omega ,\mathcal{A},P)$
and taking values in a state space $\Psi ,$ for some (step) distribution $Q$%
. Define $S_{t}(\omega )=0,$ if $t=0$ and $S_{t}(\omega )=X_{1}(\omega
)+\dots +X_{t}(\omega ),$ $t\in \mathbb{N}^{+},$ for all $\omega \in \Omega
. $ Then $S=\{S_{t}:t\in \mathbb{N}\}$ is a discrete time parameter
stochastic process with state space $\Psi $. Setting $a_{n}(\omega
)=n!X_{n}(\omega ),$ $X_{n}\thicksim Q,$ $\gamma (\omega )=1,$ and $%
B=\{0,1,...,t\},$ we have%
\begin{equation*}
S_{t}(\omega )=X(\omega )(B)=\sum\limits_{n=0}^{t}a_{n}(\omega )\frac{\gamma
(\omega )^{n}}{n!}=\sum\limits_{n=0}^{t}X_{n}(\omega ),
\end{equation*}%
and therefore, general random walks are special cases of {STMs. In
particular, setting }$\Psi =\Re ,$ and taking $Q$ the probability
distribution of a Gaussian random variable, we obtain the standard random
walk with normal steps as a special case of a GSTM.

\item Martingales: Recall that if $X$ is an iid sequence of integrable
random variables defined on the probability space $(\Omega ,\mathcal{A},P)$
with $E(X_{n})=0$, for all $n\in \mathbb{N},$ then $S_{n}=\tsum%
\limits_{i=1}^{n}X_{i}$ is a martingale, wrt minimal filtration $\mathcal{F}%
=(\mathcal{F}_{1},\mathcal{F}_{2},\dots ),$ $\mathcal{F}_{n}=\sigma
(X_{1},\dots ,X_{n}).$ Consequently, the random walk GSTM of the previous
example is also a martingale, and therefore we can use results from
martingale theory in order to study STMs, e.g., convergence theorems.

\item Autoregressive {GSTMs}: Consider an autoregressive time series, i.e.,
a stochastic process $S=\{S_{t}:t\in \mathbb{N}\},$ with%
\begin{equation*}
S_{t}=\phi S_{t-1}+\varepsilon _{t},
\end{equation*}%
for all $t\in \mathbb{N},$ where\ $\varepsilon _{t}\overset{iid}{\thicksim }%
N(0,\sigma ^{2}),$ $S_{0}=0$\ and $0\leq \phi \leq 1,$ fixed. Then iterating
backwards $t$ times, we can write%
\begin{equation*}
S_{t}=\sum\limits_{j=0}^{t-1}\phi ^{j}\varepsilon _{t-j},
\end{equation*}%
$t\in \mathbb{N}^{+},$ so that setting $\gamma (\omega )=1,$ and $%
a_{j}(\omega )=\phi ^{j}j!\varepsilon _{t-j}(\omega ),$ we have%
\begin{equation*}
S_{t}(\omega )=\sum\limits_{j=0}^{t-1}a_{j}(\omega )\frac{\gamma (\omega
)^{j}}{j!}=\sum\limits_{j=0}^{t-1}\phi ^{j}\varepsilon _{t-j}(\omega ),
\end{equation*}%
and therefore this standard example from time series modeling is indeed
another special case of GSTMs. More precisely, this is a random walk with
independent step distributions $Q_{j}$ (not identically distributed as in
the previous example), corresponding to normal distributions with means $0$
and variances $\phi ^{2j}\sigma ^{2},$ $j\in \mathbb{N}^{+}.$

\item Brownian Motion {GSTMs}: Consider a sequence of iid random variables $%
\{Z_{n}\}_{n=0}^{+\infty }$ with mean $\mu $ and variance $\sigma ^{2},$ $%
0<\sigma ^{2}<\infty $ and define the random walk $S=(S_{0},S_{1},\dots ),$
with $S_{0}=0$ and $S_{k}=\tsum\limits_{i=1}^{k}Z_{i},$ $k\in \mathbb{N}%
^{+}. $ Let $\mathcal{C}_{[0,1]}^{\Re }$ denote the collection of
continuous, $\Re $-valued functions from the interval $[0,1]$. We can build $%
\mathcal{C}_{[0,1]}^{\Re }$-valued random variables $X^{(n)}$ by defining
the function $t\longmapsto X_{t}^{(n)}(\omega )$ as follows. First consider
values for $t$ equal to a multiple of $\frac{1}{n},$ that is, for each $n\in 
\mathbb{N},$ we let%
\begin{equation*}
X_{k/n}^{(n)}=\frac{1}{\sigma \sqrt{n}}\tsum\limits_{i=1}^{k}(Z_{i}-\mu )=%
\frac{S_{k}-k\mu }{\sigma \sqrt{n}},
\end{equation*}%
for $k=0,1,2,\dots ,n$ and the random variable $X_{t}^{(n)},$ $0\leq t\leq
1, $ can be made continuous for $t\in \lbrack 0,1]$ by assuming linearity
over each of the intervals $I_{k,n}=\left[ (k-1)/n,k/n\right] ,$ that is, we
linearly interpolate the value of $X_{t}^{(n)},$ for any $t\in I_{k,n},$
based on the boundary values at $X_{(k-1)/n}^{(n)}$ and $X_{k/n}^{(n)}$. Now
note that the increment $X_{(k+1)/n}^{(n)}-X_{k/n}^{(n)}=(Z_{k+1}-\mu
)/(\sigma \sqrt{n})$ is independent of the $\sigma $-field $\mathcal{F}%
_{k/n}^{X^{(n)}}=\sigma (Z_{1},\dots ,Z_{k})$ and $%
X_{(k+1)/n}^{(n)}-X_{k/n}^{(n)}$ has zero mean and variance $%
1/n=(k+1)/n-k/n. $ When the steps $Z_{n}$ follow normal distributions, the
contruction above leads to the stochastic process $\{X_{t}^{(n)}:t\in
\lbrack 0,1]\}$ with (weak) limiting distribution known as Brownian motion
in the interval $[0,1]$ (see \cite{micheas2018theory}, Theorem 7.11). This
is once again a special case of a GSTM when $\mu =0.$
\end{enumerate}
\end{example}

As the examples above illustrate, STMs provide a general, unifying framework
that contains as special cases many important classic mathematical and
probabilistic concepts. Next we consider a generalization of Taylor's
theorem via Taylor measures.

\section{Analytic Functions and Taylor Measures}

In Example \ref{GaussianSTMs}.2, we saw that we can obtain analytic
functions as expectations of STMs. In this section, we consider the
deterministic case, where we connect the space $\mathcal{T}^{\mathcal{F}}$
with any analytic, real-valued function, by introducing an input $x\in \Re $
in the parameter $\gamma $ of a finite Taylor measure, i.e., we define $%
T_{\gamma (x),\mathbf{a}}$, for some analytic function $\gamma :\Re
\rightarrow \Re $. The following representation theorem is applicable to any
analytic function.

\begin{theorem}[Taylor Measure Representation]
\label{TMRThm}Assume that the function $f:\Re \rightarrow \Re ,$ is analytic
at a point $x_{0}\in \Re .$ Then there exists an analytic, function $\gamma
:\Re \rightarrow \Re $, a sequence $\mathbf{a}\in \Re ^{\infty },$ and a
finite Taylor measure $T_{\gamma (x),\mathbf{a}}\in \mathcal{T}^{\mathcal{F}%
} $, such that the function $f(x),$ for any $x\in \Re ,$ can be represented
as%
\begin{equation}
f(x)=T_{\gamma (x),\mathbf{a}}(\mathbb{N}).  \label{TaylorRepresentation}
\end{equation}%
Moreover, there exists an analytic function $\zeta :\Re \rightarrow \Re $
and $\mathbf{b}\in \Re ^{\infty },$ such that%
\begin{equation}
f^{n}(x)=T_{\zeta (x),\mathbf{b}}(\mathbb{N})\in \mathcal{T}^{\mathcal{F}},
\label{TaylorRepresentationPower}
\end{equation}%
for all $n\in \mathbb{N}.$
\end{theorem}

\begin{proof}
Trivially, by Taylor's theorem we can write%
\begin{equation*}
f(x)=\sum_{n=0}^{+\infty }\frac{f^{(n)}(x_{0})}{n!}(x-x_{0})^{n}=T_{\gamma
(x),\mathbf{a}}(\mathbb{N})<+\infty ,
\end{equation*}%
with $\gamma (x)=x-x_{0},$ analytic and $\mathbf{a}%
=[f(x_{0}),f^{(1)}(x_{0}), $ $f^{(2)}(x_{0}),...].$\newline
In addition, for any $n\in \mathbb{N},$ since composition of analytic
functions is an analytic function, let $g(x)=x^{n},$ an analytic function,
and write $f^{n}(x)=(g\circ f)(x)=T_{\zeta (x),\mathbf{b}}(\mathbb{N}),$ for 
$\zeta (x)=x-x_{0},$ and some $\mathbf{b}\in \Re ^{\infty }$.
\end{proof}

The Taylor measure representation of an analytic function is an immediate
and trivial consequence of Taylor's theorem, but we note that the
representation is not unique. In particular, if $f(x)=T_{\gamma (x),\mathbf{a%
}}(\mathbb{N}),$ where $\gamma $ is analytic with $\gamma (x)\neq x-x_{0},$
then we can always find $\mathbf{c}\in \Re ^{\infty }$, such that 
\begin{equation}
f(x)=T_{\gamma (x),\mathbf{a}}(\mathbb{N})=T_{x-x_{0},\mathbf{c}}(\mathbb{N}%
),  \label{RepresNotUnique}
\end{equation}%
provided that $\sum_{n=0}^{+\infty }\frac{a_{n}}{n!}<+\infty .$ To see this,
since $\gamma (x)^{n}=T_{x-x_{0},\mathbf{b}}(\mathbb{N}),$ for some $\mathbf{%
b}\in \Re ^{\infty }$, we use equation (\ref{TaylorRepresentationPower}),
and write%
\begin{eqnarray*}
f(x) &=&T_{\gamma (x),\mathbf{a}}(\mathbb{N})=\sum_{n=0}^{+\infty }a_{n}%
\frac{\gamma (x)^{n}}{n!}=\sum_{n=0}^{+\infty }\frac{a_{n}}{n!}T_{x-x_{0},%
\mathbf{b}}(\mathbb{N}) \\
&=&\sum_{n=0}^{+\infty }\frac{a_{n}}{n!}\sum_{l=0}^{+\infty }\frac{b_{l}}{l!}%
(x-x_{0})^{l}=\sum_{k=0}^{+\infty }\frac{c_{k}}{k!}(x-x_{0})^{k},
\end{eqnarray*}%
where%
\begin{equation*}
c_{k}=b_{k}\sum_{n=0}^{+\infty }\frac{a_{n}}{n!},
\end{equation*}%
provided that the series converges. As a result, one can assume wlog the
representation of equation (\ref{RepresNotUnique}) for any analytic
function. Furthermore, since $\gamma $ is analytic, it is continuous, and
therefore, by construction $f\in \mathcal{C}_{\Re }^{\Re },$ where $\mathcal{%
C}_{\Re }^{\Re }$ denotes the space of real valued, continuous functions
from $\Re .$

Now consider the space of analytic functions%
\begin{equation}
\mathcal{G}_{\mathbb{N}}=\{f:f(x)=T_{x-x_{0},\mathbf{a}}(\mathbb{N}),\text{
for some }\mathbf{a}=[a_{0},a_{1},...]\in \Re ^{\infty },T_{x-x_{0},\mathbf{a%
}}\in \mathcal{T}^{\mathcal{F}}\},  \label{TaylorFunctionSpace}
\end{equation}%
with $\mathcal{G}_{\mathbb{N}}\subset \mathcal{C}_{\Re }^{\Re }.$ We provide
some insight on the structure of $\mathcal{G}_{\mathbb{N}},$ in the
following.

\begin{lemma}
\label{AlgebraFuncs}The space $\mathcal{G}_{\mathbb{N}}$\ is an algebra of
functions that includes constant functions and separates points, i.e., it is
a vector space of functions that is also closed under pointwise
multiplication, and for each $x\neq y$ there is a function $f\in \mathcal{G}%
_{\mathbb{N}},$ with $f(x)\neq f(y).$
\end{lemma}

\begin{proof}
Let $f_{1},f_{2}\in \mathcal{G}_{\mathbb{N}},$ with%
\begin{equation*}
f_{k}(x)=\sum\limits_{n\in \mathbb{N}}a_{kn}\frac{(x-x_{0})^{n}}{n!},
\end{equation*}%
for $a_{kn}\in \Re ,$ $k=1,2$, $n\in \mathbb{N}.$ Since $\mathcal{T}^{%
\mathcal{F}}$ is a Hilbert space, $\mathcal{G}_{\mathbb{N}}$ is a vector
space. Moreover, we can write%
\begin{eqnarray*}
f_{1}(x)f_{2}(x) &=&\left( \sum_{n=0}^{+\infty }a_{1n}\frac{(x-x_{0})^{n}}{n!%
}\right) \left( \sum_{n=0}^{+\infty }a_{2n}\frac{(x-x_{0})^{n}}{n!}\right) \\
&=&\sum_{n,k=0}^{+\infty }a_{1n}a_{2k}\frac{(x-x_{0})^{n+k}}{n!k!}%
=\sum_{l=0}^{+\infty }b_{l}\frac{(x-x_{0})^{l}}{l!}\in \mathcal{G}_{\mathbb{N%
}},
\end{eqnarray*}%
for some $b_{l}\in \Re $ and any $x\in \Re $, so that $\mathcal{G}_{\mathbb{N%
}}$ is closed under pointwise multiplication. Clearly, for any constant $%
c\in \Re ,$ take $\gamma =1,$ and $\mathbf{a}(x)=[0,c,0,...],$ so that $%
c=T_{\gamma ,\mathbf{a}(x)}(\mathbb{N})\in \mathcal{G}_{\mathbb{N}}.$ Now
take arbitrary $x,y\in \Re $, with $x\neq y,$ choose any 1 to 1 analytic
function $f,$ and $\mathbf{a}(x)=[0,1,0,...],$ so that $f(x)=T_{f(x),\mathbf{%
a}(x)}(\mathbb{N})\in \mathcal{G}_{\mathbb{N}}$, is such that $f(x)\neq
f(y), $ and therefore, $\mathcal{G}_{\mathbb{N}}$ separates points.
\end{proof}

Finding dense subsets in spaces of functions is a crucial, well known and
studied problem for spaces of continuous functions. Recall the classic
Stone-Weierstrass theorem from functional analysis (e.g., \cite{Dudley2004},
Theorem 2.4.11), which requires properties for the dense space as in Lemma %
\ref{AlgebraFuncs}, and a compact Hausdorff space $K,$ so that the set of
functions is dense in $\mathcal{C}_{\Re }^{K}$. The space of functions $%
\mathcal{G}_{\mathbb{N}}$ satisfies almost all requirements ($\Re $ is
Hausdorff), but $\Re $ is not compact. Therefore, in order to apply the
Stone-Weierstrass theorem for $\mathcal{G}_{\mathbb{N}}$ we consider the
restriction of $\mathcal{G}_{\mathbb{N}}$ to analytic functions defined over
a compact subset $K\subset \Re $, i.e., let%
\begin{equation*}
\mathcal{G}_{\mathbb{N}}^{K}=\{f:f(x)=T_{x-x_{0},\mathbf{a}}(\mathbb{N}),%
\text{ }x\in K\subset \Re ,\text{ }\mathbf{a}=[a_{0},a_{1},...]\in \Re
^{\infty },T_{x-x_{0},\mathbf{a}}\in \mathcal{T}^{\mathcal{F}}\},
\end{equation*}%
so that $\mathcal{G}_{\mathbb{N}}^{K}$ is dense in $\mathcal{C}_{\Re }^{K}$
by Stone-Weierstrass, using $d_{sup}(f,g)=\underset{x\in K}{\sup }%
|f(x)-g(x)| $ norm.

Now let $\mu $ denote Lebesgue measure in $(\Re ,\mathcal{B}(\Re )),$ and
consider the usual inner product%
\begin{equation}
(f_{1},f_{2})=\int\limits_{\Re }f_{1}(x)f_{2}(x)d\mu (x).
\label{InnerProduct1}
\end{equation}%
Then all the usual results we are familiar with from functional analysis
(e.g., see \cite{Dudley2004}, Chapter 5) also hold for the space $\mathcal{G}%
_{\mathbb{N}},$ with some additional assumptions in order to handle
unbounded functions and the integrals involved, e.g., for any $f\in \mathcal{%
G}_{\mathbb{N}}$ we require that $\int\limits_{\Re }|f|^{p}d\mu <+\infty ,$ $%
1\leq p<+\infty $. We denote this restricted space by $\mathcal{G}_{\mathbb{N%
}}^{\mathcal{L}^{p}}.$ Then using the $\mathcal{L}^{p}$-norm, $\left\Vert
f\right\Vert _{p}=\left( \int\limits_{\Re }|f|^{p}d\mu \right) ^{\frac{1}{p}%
},$ we have immediately that $\mathcal{G}_{\mathbb{N}}^{\mathcal{L}%
^{p}}\subset \mathcal{L}^{p}(\Re ,\mathcal{B}(\Re ),\mu ),$ i.e., $\mathcal{G%
}_{\mathbb{N}}^{\mathcal{L}^{p}}$ is a subset of the Lebesgue $\mathcal{L}%
^{p}$ space, and thus inherits its properties.

\section{Concluding Remarks}

We introduced and studied properties of a novel collection of signed
measures, Taylor measures. Whilst spaces of general signed measures are
typically studied using total variation norm, the latter was not useful in
proving desirable properties of the space $\mathcal{T}^{\mathcal{F}}$.
Instead, we proposed a new inner product $\rho \left( .,.\right) $ which
allowed us to prove that the space $\mathcal{T}^{\mathcal{F}}$ is a Hilbert
and a Polish space. As a result, the proposed space $\mathcal{T}^{\mathcal{F}%
}$ is well behaved, meaning that all important concepts and desirable
properties from real analysis and measure theory are present, including the
existence of orthonormal bases, dense subsets, and reproducing formulas of
elements of $\mathcal{T}^{\mathcal{F}},$ to name but a few.

We presented first applications of the proposed measures including the
creation of a new collection of discrete probability measures, the positive
and negative Taylor probability measures. More importantly, the positive
Taylor probability measure allowed us to describe any discrete probability
measure and its discrete density, under the standard assumption of absolute
continuity wrt counting measure.

Moreover, we defined stochastic versions of Taylor measures, and illustrated
via examples, that they provide a unifying framework for many important
concepts from statistics and probability theory, including, Brownian motion,
martingales, random walks and time series. STMs also provide us with a way
of estimating the values of a Taylor measure using statistical modeling and
simulation.

We further used $\mathcal{T}^{\mathcal{F}}$ to create a space of functions
that allowed us to show that $\mathcal{T}^{\mathcal{F}}$ is a generalization
of Taylor's classic theorem for analytic functions. Generalizations to the
multivariate case, as well as extensions to measurable functions $\gamma $
are also under current investigation. In addition, the space $\mathcal{G}_{%
\mathbb{N}}^{\mathcal{L}^{p}}$ serves as the starting point for the
investigation of the continuous analog to Theorem \ref{DTaylorPMRepres}.
Furthermore, one can define the Taylor integral, which leads to the
definition of Taylor measure differential equations, and more importantly,
stochastic (partial) Taylor measure differential equations. These initial
results were not included in this paper, since our purpose herein was to
introduce the measures and their first applications in mathematics and
probability theory.

These are subjects of great interest, in further illustrating the importance
of $\mathcal{T}^{\mathcal{F}}$, and will be presented elsewhere.

\section{Statements and Declarations}

The author has no financial or non-financial conflicts of interest that are
directly or indirectly related to this work submitted for publication.

\bibliographystyle{plain}
\bibliography{TaylorMeasure}

\end{document}